\documentclass[12pt,reqno]{amsart}
\newcounter{defcounter}
\usepackage{amsfonts}
\usepackage{amsmath}
\usepackage{array}
\usepackage{enumerate}
\usepackage{graphicx}
\usepackage{amssymb}
\usepackage{amsthm}
\usepackage{verbatim,amscd,amsmath,amsthm,amstext,amssymb,amsfonts,latexsym,lscape,rawfonts
}
\usepackage[arrow,matrix,curve]{xy}
\usepackage{xcolor}
\newenvironment{myequation}{%
\addtocounter{equation}{-1}
\refstepcounter{defcounter}

\begin{equation}}
{\end{equation}}

\newtheorem{theorem}{Theorem}
\newtheorem{corollary}[theorem]{Corollary}
\newtheorem{lemma}[theorem]{Lemma}
\newtheorem{proposition}[theorem]{Proposition}
\theoremstyle{definition}

\newtheorem{remark}{Remark}

\newcommand{\C}{\mathbb{C}}
\newcommand{\D}{\mathbb{D}}

\newcommand{\Z}{\mathbb{Z}}

\newcommand{\be}{\begin{equation}}
\newcommand{\ee}{\end{equation}}

\begin{document}

\title[Theta functions on tube domains]{Theta functions on tube domains}
\author{Josef F.~Dorfmeister}
\address{Fakult\"{a}t f\"{u}r Mathematik, TU M\"{u}nchen, Boltzmannstr. 3,
D-85747 Garching, Germany}
\email{dorfm@ma.tum.de}
\author{Sebastian Walcher}
\address{Mathematik A, RWTH Aachen, D-52056 Aachen, Germany}
\email{walcher@matha.rwth-aachen.de}

\maketitle

\begin{abstract} 
We discuss generalizations of classical theta series, requiring only some basic properties of the classical setting. As it turns out, the existence of a generalized theta transformation formula  implies that the series is defined over a quasi-symmetric Siegel domain. In particular the exceptional symmetric tube domain does not admit a theta function.

\end{abstract}

\section{Introduction}

{The classical theta series
\begin{equation}\label{thetaone}
\theta(z,u)=\sum_{l\in\mathbb Z}\exp\left(\pi i\cdot  z l^2+2\pi i\cdot  lu\right)
\end{equation}
for $z\in \mathbb H$ (the upper half plane) and $u\in\mathbb C$ has been generalized by various authors to the case of several variables; see e.g. Mumford \cite{Mum}, Chapter II for an overview. A far-reaching generalization was given in Krieg \cite{KrDiss}, Ch.~4, \S1, encompassing Hermitian matrices over $\mathbb R$, $\mathbb C$ and Hamilton's quaternions, respectively.
In a yet more general approach, Resnikoff \cite{Res} and Dorfmeister \cite{DoInv} defined theta series for any formally real Jordan algebra $\mathfrak A$ that does not admit the 27-dimensional exceptional simple Jordan algebra as a direct summand. In  \cite{DoInv}, for instance, a series representation of the form
\begin{equation}\label{thetaSD}
\theta(z,u)=c\cdot\sum_{l\in\Lambda}\exp\bigl(\pi i\cdot \tau(\phi(z)l+2u,Al)\bigr)
\end{equation}
is given, where $z$ lies in the complexification of $\mathfrak A$, $u$ lies in the complexification of a nontrivial real vector space $\widehat{U}$, $\phi$ is a Jordan homomorphism from $\mathfrak A$ to the endomorphisms of $\widehat{U}$, $\tau$ is a suitable trace form of $\mathfrak{A}$ and $A$ an endomorphism of $\widehat{U}$ with certain properties.

It seemed to be a commonly accepted fact that there exists no theta series for the exceptional symmetric tube domain over the exceptional Jordan algebra. 

In the present paper we discuss the geometric and algebraic structures underlying theta series. We start from a series expansion -- quite similar to the classical theta series -- which satisfies the usual ``translational perio\-dicity" conditions \eqref{theta1} and \eqref{theta2}. This series is defined on a domain $\mathcal{D}$ of the form 
$\mathcal{H}(Y ) \times \widehat{U}^\C$, where  $Y$ is an open, regular convex cone in some real vector space $V$,
$\mathcal{H}(Y) = V \oplus i Y$ is a tube domain and $\widehat{U}$ is a real vector space. Conversely it is readily seen that holomorphic functions satisfying  \eqref{theta1}, \eqref{theta2} admit a series expansion as given, up to a multiplicative factor depending on the first variable only. 
Taking this factor as a constant (as we will do), the series admits a further periodicity property \eqref{theta3}.\\
The series construction involves a linear map $\psi$ from the real vector space $V$ into the self-adjoint linear transformations of $\widehat{U}$ (relative to some positive definite symmetric bilinear form), with the image containing positive definite maps. Up to this point the tube domain and the linear map $\psi$ are the main ingredients. \\
We continue with a discussion of automorphisms of $\mathcal D$ and their transformation behavior with respect to the theta series. In particular, the linear automorphisms of $\mathcal D$ satisfy an identity \eqref{lineartheta}.\\
 Next we add a crucial property \eqref{theta-transformation} that is akin to the classical theta transformation formula. We will not discuss ``partial" transformation formulas (like e.g. in \cite{Dieck}) in this paper and therefore assume the existence of a map $s(z,u) = (j(z), J(z)u)$ in such a transformation formula
which satisfies $j\circ j = id$ and  $dj(ie) = -I$ for the differential, where $ie$ is a base point in the tube domain $\mathcal{H}(Y)$. We prove that this requirement already implies that the tube domain is symmetric, thus $Y$ is the positivity cone of some formally real Jordan algebra $\mathfrak A$ on $V$, and $-j(x)=x^{-1}$ is the inverse in $\mathfrak A$. Moreover one can then assume that the cone $Y$ is self-dual relative to $\sigma$. \\
We proceed to show that the Jordan algebra $\mathfrak A$ is special, and thus give a proof of the above mentioned commonly accepted fact that no theta series exists over the exceptional tube domain. Moreover, the theta series is necessarily of the type introduced in \cite{Res} and \cite{DoInv}.

\section{ The basic set-up }\label{basicsec}
We first list the fundamental ingredients, with  further properties} to be added in the sequel.
\begin{enumerate}[(i)]

\item Given are a real vector space $\widehat U$ with complexification $U=\widehat U^\mathbb C$, a real vector space $\widetilde{V}$ and a connected and open domain $\D \subset \widetilde{V}^\C \times U$. Moreover we assume that $\mathbb D$ is invariant with respect to all translations $(z,u)\mapsto (z + a,u), a \in \widetilde{V}$.

\item Let $\Lambda$ be a lattice of full rank in $\widehat U$; moreover we assume that $(z,u+m)\in\mathbb D$ whenever $(z,u)\in \mathbb D$ and $m\in\Lambda$.
\item Let $\rho$ be a positive definite symmetric bilinear form on $\widehat U$. Define the dual lattice $\Lambda^\rho$ with respect to $\rho$ as usual, thus $x\in \Lambda^\rho$ if and only if $\rho(\Lambda,x)\subseteq \mathbb Z$. We extend $\rho$ to a symmetric $\mathbb C$-bilinear form on $U$ which will be denoted by the same symbol.

\item Let $\psi: \widetilde V^\mathbb C\to{\rm End}_\mathbb C(U)$ be a $\mathbb C$-linear map. Define $\widetilde Y\subseteq \widetilde V$ by:
\[
y\in\widetilde Y \text{ if and only if }\rho(\psi(y)l,l)>0 \text{ for all nonzero  } l\in\Lambda^\rho.
\]
We require  that $\widetilde Y\not=\emptyset$.
\end{enumerate}

\medskip\noindent
We will consider the series
\be \label{generalthetahat}
{\theta} (z,u) ={\theta}_\Lambda (z,u):= \sum_{l \in \Lambda^\rho} \exp\bigl(\pi i\cdot \rho( \psi(z)l + 2u, l)\bigr).
\ee
(The relevant lattice $\Lambda, \Lambda^\rho$ etc. will be included as a subscript when necessary.) As a motivation for introducing this series, note that all the classical examples (including those mentioned in the Introduction) fit within this framework.

\subsection{Properties of the series expansion}
We first state some basic properties of the series \eqref{generalthetahat}. The proofs are straightforward variants of the classical ones.
\begin{proposition} 
\begin{enumerate}[(a)]

\item The series \eqref{generalthetahat} converges absolutely and uniformly on any compact subset of $\mathbb D_0:=(\widetilde V+i\widetilde Y)\times U$.

\item If the series converges at $(z,u)$ and $m\in\Lambda$ then the series converges at $(z,u+m)$ and one has the identity
\begin{myequation}\label{theta1}
{\theta} (z,u+m) ={\theta} (z,u).
\end{myequation}

\item  If the series converges at $(z,u)$ and $n\in\Lambda^\rho$ then the series converges at $(z,u+\psi(z)n)$ and one has the identity
\begin{myequation}\label{theta2}
{\theta} (z,u+\psi(z)n) =\exp(-\pi i\cdot \rho(\psi(z)n+2u,n))\cdot{\theta} (z,u).
\end{myequation}
\end{enumerate}
\end{proposition}

In the following we will always assume that $\mathbb D\supseteq\mathbb D_0$.\\

The conditions \eqref{theta1} and \eqref{theta2} reflect standard properties of classical theta functions. In turn, they determine functions which are defined by the series \eqref{generalthetahat} up to a $z$-dependent factor, as will be shown next.

\begin{theorem} Let $\widehat \theta$ be holomorphic on the open and connected set $\mathbb D\supseteq\mathbb D_0$, satisfying the identities  \eqref{theta1} and \eqref{theta2}. Then there exists a holomorphic function $\widehat\theta_0$, which is defined on the image of $\mathbb D$ under the projection $(z,u) \mapsto z$ to the first component such that
\[
\widehat\theta(z,u)=\widehat\theta_0(z)\cdot\sum_{l \in \Lambda^\rho} \exp( \pi i\cdot \rho( \psi(z)l + 2u, l)).
\]
\end{theorem}

\begin{proof}
Due to periodicity with respect to the second entry, the holomorphic function $\widehat {\theta} : \D \rightarrow \C$ admits a Fourier expansion of the form
\be \label{basicexpansion}
\widehat{\theta} (z,u) = \sum_{l \in \Lambda^\rho}\widehat {\theta}_l (z) e^{2 \pi  i \rho (u,l)},
\ee
where the series converges uniformly and absolutely on all compact subsets of $\D$ and the coefficient functions $\widehat\theta_l(z)$ are holomorphic in $z$.\\
Proceeding with \eqref{theta2} we obtain
\[
\begin{array}{rcl}
\widehat{\theta} (z, u + \psi (z)m) &=& \sum_{l \in \Lambda^\rho} \widehat{\theta}_l(z) 
e^{2 \pi i \rho (u + \psi (z)m,l)} \\
&=& e^{-i \pi  \rho (2u + \psi (z)m,m)} \widehat{\theta}(z,u)  \\
&=&  e^{-i \pi  \rho ( \psi (z)m,m)} \sum_{l \in \Lambda^\rho} \widehat{\theta}_l(z) 
e^{2 \pi i \rho (u, -m + l)} \\
 &=& e^{-i \pi  \rho ( \psi (z)m,m)}  \sum_{l \in \Lambda^\rho} \widehat{\theta}_{l + m}(z) 
e^{2 \pi i \rho (u, l)} .
\end{array}
\]
Comparing the second to the last term in this sequence of equalities we obtain for all $m,l \in \Lambda^\rho$
$$ \widehat{\theta}_l(z) e^{2 \pi i \rho ( \psi (z)m,l)} = e^{-i \pi  \rho ( \psi (z)m,m)}  \widehat{\theta}_{l + m}(z) $$
Putting $l =0$ we infer
$$\widehat{\theta}_m(z)  = \widehat{\theta_0}(z)  e^{i \pi  \rho ( \psi (z)m,m)} $$
The claim follows.
\end{proof}

\begin{remark}
At this point the choice of the holomorphic function $\widehat{\theta}_0$  is completely free. In the classical cases $\widehat{\theta}_0$  is a constant (possibly depending on the lattice, e.g. as in Resnikoff \cite{Res}). \\
In the following we  assume that  $\widehat\theta_0=1$. As a consequence, our version of the theta transformation formula (see \eqref{theta-transformation} below) will contain a  factor $c_\Lambda$ to account for the difference between the definitions in \cite{Res}  and \cite{DoInv}.
\end{remark}

\subsection{Properties of $\rho$ and $\psi$}

Up to this point we did not specify any further properties of the symmetric bilinear form $\rho$, and we did not address uniqueness questions concerning the map $\psi$ from Condition (iv). We will attend to this now.

\begin{remark}
\begin{enumerate}[(a)]

\item We will assume that $\psi$ is injective. This involves no loss of generality, since the series expansion \eqref{generalthetahat} is unaffected by terms in the kernel of $\psi$.  

\item Since $\psi$ is $\mathbb C$-linear, it is completely determined by its effect on $\widetilde V$.
\end{enumerate}
\end{remark}

\begin{lemma}\begin{enumerate}[(a)]
\item Let us put
\[
\widetilde Y:=\left\{x\in\widetilde V;\,\rho(\psi(x)u,\,u)>0\text{  for all  }u\in \widehat{U} \setminus\{0\}\right\}.
\]
Then $\widetilde Y$ is a nonempty, open and convex cone in $\widetilde V$. 

\item There is a direct sum decomposition $\widetilde V=V_0\oplus V_1$ of vector spaces such that
\[
\widetilde Y=Y_0\times V_1
\]
with a regular open convex cone $Y_0\subseteq V_0$, where regularity means that $Y_0$ does not contain any full line. 

\item For all $x=x_0+x_1\in V_0\oplus V_1$ and $u\in \widehat{U}$ one has
\[
\rho(\psi(x)u,\,u)=\rho(\psi(x_0)u,\,u).
\]
\end{enumerate}
\end{lemma}

\begin{proof} For part (a) we note that $\widetilde{Y}$  is clearly open and convex, and it is nonempty by assumption (iv). As for part (b), $V_1$ is the union of all full straight lines in $\widehat Y$ (which is a vector space due to convexity), and $\widetilde{Y} \,{\rm mod}\, V_1$ is a regular cone.
\end{proof}

We now turn to questions of uniqueness.
\begin{lemma}\begin{enumerate}[(a)]
\item Let $\psi$ and $\psi^*$ be $\mathbb C$-linear maps from $\widetilde V^\mathbb C$ to ${\rm End}_\mathbb C(U)$ such that $\rho(\psi(z)u,\,u)=\rho(\psi^*(z)u,\,u)$ for all $z\in\widetilde V^\mathbb C$ and all $u\in U$. Then $(\psi-\psi^*)(x)$ is skew-symmetric with respect to $\rho$ for all $x\in \widetilde V$. Conversely, if $\mu :\,\widetilde V^\mathbb C\to{\rm End}_\mathbb C(U)$ is $\mathbb C$-linear and $\mu(x)$ is skew-symmetric for all $x\in\widetilde V$ then  $\rho((\psi+\mu)(z)u,\,u)=\rho(\psi(z)u,\,u)$ for all $z\in\widetilde V^\mathbb C$ and all $u\in U$.
\item Given $\psi$ as in Condition (iv), there is a unique $\psi_0:\,\widetilde V^\mathbb C\to{\rm End}_\mathbb C(U)$ such that  $\rho(\psi(z)u,\,u)=\rho(\psi_0(z)u,\,u)$ for all $z\in\widetilde V^\mathbb C$ and all $u\in U$ and $\psi_0(x)$ is symmetric (i.e. self-adjoint) with respect to $\rho$ for all $x\in\widetilde V$.
\end{enumerate}
\end{lemma}

\begin{proof} Part (a) is obvious. For part (b) define 
\[
\psi_0(x):=\frac12 (\psi(x)+\psi^\rho(x)),\quad x\in\widetilde V.
\]
\end{proof}

\begin{remark}\label{hermmaprem} 
In view of this Lemma, we may (and, from now on, will) assume that $\psi(x)$ is self-adjoint with respect to $\rho$ for all $x\in \widetilde V$. Moreover, elements of $V_1$ have no effect on the series expansion \eqref{generalthetahat}; hence we may (and, from now on, will) assume that $\widetilde V=V_0$, i.e. we will continue to require injectivity for the self-adjoint transformtion $\psi$. We abbreviate $V:= V_0$ and $Y:=Y_0$ in what follows.
\end{remark}

\begin{remark} \label{alternativerem}
In Conditions (i) and (iv) we started from a real vector space $\widetilde V$ and then turned to its complexification; the crucial condition concerning $\widetilde V$ is the requirement that $\widetilde Y\not=\emptyset$. We sketch here an approach that starts from a complex vector space $W$ and a linear map from $W$ to ${\rm End}_\mathbb C(U)$, with the same final result. \\
The crucial requirement that we impose here is as follows:
There exists  $z \in W$ such that for all $l \in \Lambda^\rho$ we have: $\Im (\rho ( \psi(z)l,l)) >0,$ where $\Im v$ denotes the imaginary part of $v.$\\
This condition simply assures convergence of the series \eqref{generalthetahat} for $(z,\,u)$, with $u\in U$ arbitrary. Indeed, assume that $\Im (\rho ( \psi(z)l^*,l^*)) \leq0$ for some $l^*\in \Lambda^\rho$. Then the sequence of summands labelled by elements in $\Z l^*$ will not converge to $0$. \\
Now let
$\widehat{Y}$ denote the subset of all those $w\in W$ for which 
$\rho(\psi(w)v,\,v)>0$ for all nonzero $v\in U$.
Then $\widehat{Y}$ is a convex cone, which is contained in an open subset of some  real subspace $\widetilde V$ of $W$. {Moreover, since $\psi$ is injective, this cone is regular.} One verifies $\widetilde V^\mathbb C=W$; from here on Remark \ref{hermmaprem} applies.
\end{remark}

Keeping the assumptions and the notation introduced above, we choose a {\em base point} $e \in {Y}$. Since $\psi(e)$ is positive definite, there exists an endomorphism $T$ of $\widehat{U}$ such that $\psi(e) = T^\rho T$.

\begin{lemma}
With $e \in{Y}$ and $\psi(e) = T^\rho T$ as above, 
 $\check{\rho}(u,v):= \rho(Tu,Tv) $ defines a symmetric bilinear form and 
$\check{\psi} (z) = \psi^{-1}(e) \psi(z)$ defines a $\mathbb C$-linear map that is self-adjoint with respect to $\check\rho$ for all $x \in V$ and positive definite for all  $x \in {Y}$, with $\check\psi(e)=I$, the identity.
\end{lemma}

\begin{proof}
Consider $\check{\rho} (\check{\psi}(x)u,v)  = \rho(T \psi (e)^{-1} \psi(x) u , Tv)
=  \rho (\psi (x) u,v) = \rho (u, \psi (x) v) = \rho (Tu, T \psi(e)^{-1} \psi(x)v) =
 \check{\rho} (u, \check{\psi}(x)v)$ for all $x \in {V}$. From this the other statements follow easily.
\end{proof}

\begin{remark} Replacing the lattice $\Lambda$ by $(T^\rho T)^{-1}\Lambda$, one can  rewrite the series  \eqref{generalthetahat} in terms of $\check\rho$. Thus we can, and from now on will, assume that for $\rho$ and $\psi$ there is a base point $e\in Y$ such that $\psi(e)=I$, in addition to all the other properties noted in Remark \ref{hermmaprem}. 
\end{remark}

\subsection{A natural Siegel domain}
We now choose a positive definite symmetric bilinear form $\sigma$ on $V$ and define
\begin{equation}   \label{dualcone}
Y^\sigma = \{ x \in V; \sigma (x,y) >0 \hspace{2mm} \mbox{ for all} \hspace{2mm} y \in \overline{Y} \setminus \{0\} \}.
\end{equation}
This is a regular convex cone in $V$, usually called the  {\em dual cone} (or, more precisely, the {\em $\sigma-$dual cone}) of $Y$.

Next we define the map $S: U \times U \rightarrow V^\C$  by
\be
\sigma(S(u,v),z) = \rho(\psi(z)u,\overline v) \hspace{2mm} \mbox{ for all } \hspace{2mm} z\in V^\C.
\ee
(In passing we note that the symmetric bilinear form $\rho$ on $\widehat U$ may also be extended to a Hermitian form $\rho_H$ on $U$ by setting $\rho_H(u,v):=\rho(u,\overline v)$. Replacing $\rho$ by $\rho_H$ would not change the series \eqref{generalthetahat}.) 
\begin{proposition} \label{properties ofS}
The map $S$ has the following properties.
\begin{itemize}
\item $S$ is $\C-$linear in the variable u.

\item $S(u,v) = \overline{S(v,u)},$ in particular, $S$ is $\C-$antilinear in  v.

\item $ S(u,u) \in \overline{ Y^\sigma} $ for all $u \in U$.

\item $S(u,u) = 0 \Rightarrow u = 0.$

\item The family of all $S(l,l), l\in \Lambda^\rho,$ generates $V$.
\end{itemize}
\end{proposition}

\begin{proof}
The first two statements are immediate from the definitions.  The third claim follows, in view of the definition of the dual cone,  immediately from $\sigma(S(u,u),y) = \rho( \psi(y)u,\overline u) \geq 0$ for all $u\in U$ and all $y \in Y$.
Next note that $ 0 = \sigma( S(u,u), y) = \rho(\psi(y)u,\overline u)$ for some $y \in Y$ implies $u = 0$, since $\psi(y)$ is positive definite. Finally, let $x\in V$ such that $\sigma(S(l,l), x)=0$ for all $l\in\Lambda^\rho\subseteq \widehat{U}$. Then by definition $\rho(\psi(x)l,l)=0$ for all $l$, hence $\psi(x)=0$ since $\psi(x)$ is self-adjoint, whence $x\in {\rm Ker}(\psi)=\{0\}$.
\end{proof}

Using the sesquilinear form $S$ we can define a Siegel domain in a natural manner. Since we have assumed a priori that $U$ is non-trivial, this will be a Siegel domain of type II.

Retaining the assumptions and the notation introduced in the previous sections,  a {\em  Siegel domain of type II} is defined as follows:
$$ \mathcal{D} (Y^\sigma, S) = \{ (z,u) \in V^\C \times U; \Im(z) - S(u,u) \in Y^\sigma\};$$
see \cite{Do;HomSieg,Do;Habil} and the references quoted there.
The intersection of this domain with $V^\C$ is the {\em half-space} (also called {\em tube domain})
$$ \mathcal{H} (Y^\sigma) = V \oplus iY^\sigma.$$
Note that $\mathcal{D} (Y^\sigma, S)$ is an open and connected subset of $V^\C \times U$ and $\mathcal{H} (Y^\sigma)$ is an open and connected subset of $V^\C$.

\subsection{A further periodicity property}

We recall that $\theta (z,u)$, by its definition \eqref{generalthetahat}, is periodic in the variable $u$ relative to the lattice $\Lambda$ in $\widehat{U} \subset U$. Moreover, $\theta(z,u)$ is also periodic in $z$ relative to some lattice in the real vector space $V$.

\begin{theorem} The $\mathbb Z$-module
\be \label{defineL}
\mathcal{L} = \{ x \in V; \rho ( \psi(x)l,l) \in 2 \Z  \hspace{2mm} \mbox{ for all} \hspace{2mm} l \in \Lambda ^\rho\}
\ee
is  a lattice in $V$, and the identity
\begin{myequation} 
\label{theta3}
\theta(z+k ,u) = \theta (z,u)
\end{myequation}
holds for all $ k \in \mathcal{L}$.
\end{theorem}

\begin{proof}
Periodicity \eqref{theta3} follows immediatly from the defining equation \eqref{generalthetahat}.
It remains to prove that $\mathcal{L}$ is discrete. Assume that there exists some sequence $k_n$ in $\mathcal{L}$ which converges in $V$ towards some $w\in V$.
Then for every $l\in \Lambda^\rho$ the sequence $\sigma(k_n, S(l,l))$ consists of integers and thus stabilizes. Therefore up to finitely many elements the sequence $\{ k_n\}$ is contained in a hyperplane perpendicular to $S(l,l)$.
Since the family of $S(l,l), l\in \Lambda^\rho,$ generates $V$, the intersection of the corresponding perpendicular hyperplanes consists of the point $0$  only. 
\end{proof}

\section{Automorphisms of Siegel domains and transformations of theta functions }

\subsection{General remarks }

A look at the series defining  $\theta(z,u)$ shows that this function actually is defined and holomorphic with compactly convergent series representation
for all $(z,u)$ on the open domain 
\begin{equation}\label{Tdef}
\mathcal{T}(Y) = (V \oplus iY) \times U.
\end{equation}
The Siegel domain $ \mathcal{D} (Y^\sigma, S)$ is contained in, but not equal to  $\mathcal{T}(Y^\sigma) $. However, as will be seen, certain automorphisms of Siegel domains naturally correspond to  transformations of $ \mathcal{T}(Y)$ that are distinguished with respect to the function $\theta$. From this perspective it is appropriate to briefly review automorphisms of (not necessarily homogeneous or even symmetric) Siegel domains.

In general, Siegel domains of type II admit five basic types of automorphisms; see \cite{Do;Habil} for details and proofs.\\
The simplest type are the translations $(z,u) \mapsto (z + a,u)$ which are automorphisms of the Siegel domain for any $a \in V$.
For $\theta(z,u)$ translations by elements of $\mathcal{L}$ yield a nice transformation behaviour.\\
The next type are affine transformations involving primarily the variable $u$.
It turns out that the transformation 
\[
(z,u) \mapsto ( z + 2iS(u,d) + i S(d,d), u + d)
\] 
is for every $d \in U$ an automorphism of $ \mathcal{D} (Y^\sigma, S)$ and also of $\mathcal{T} (Y^\sigma)$.\\
For $\theta (z,u)$ the transformations 
$(z,u) \mapsto (z, u+l), l \in \Lambda^\rho,$ yield a nice transformation behaviour and are actually employed in the definition of a theta function.\\
The third type of transformations are the linear transformations which we will investigate in more detail below. In general, there may be only few linear automorphisms of a Siegel domain. By Theorem \ref{linearauto} below, the same holds for theta functions.\\
The remaining basic transformations are non-affine.
One type primarily involves the variable $u$ and the other one involves primarily the variable $z$.\\
The first of these two types  (assuming existence) has been described quite explicitly in \cite{Do;AmJMath} and \cite{Rot} for the case of Siegel domains.
We are not aware of any such non-linear transformations appearing in the context of  theta functions, although they always do exist for symmetric Siegel domains, which underlie the classical theta functions.\\
Non-linear transformations of the last type do not always exist for Siegel domains of type II, but they always exist for homogeneous tube domains. With respect to theta functions the existence of such transformations is necessary to even state (full and partial) theta transformation formulas.

It is probably due to the original classical motivation that the transformations which occur in the same form for both theories, Siegel domains and theta functions, actually are extensions of automorphisms of the tube domain $ \mathcal{H} (Y^\sigma)$ for self-dual cones $Y=Y^\sigma$. Note that such extensions exist by \cite{Do;HomSieg, Do;Habil}.

\subsection{The transformation behavior of $\theta$ under linear substitutions}
In this subsection we discuss linear transformations of the type
\[
(z,\,u)\mapsto (Bz,\,\widehat Bu)
\]
with linear automorphisms $B$ of $V$ and $\widehat B$ of $\widehat U$ (which will be extended linearly to the respective complexifications).

In \cite{Res}, Thm.~2.3 one finds a formula which translates in our setting to

\begin{myequation} \label{lineartheta}
\theta_{\widehat {W}\Lambda} (W z, \widehat{W} u) = \theta_{\Lambda}(z,u)
\end{myequation}
for all $z\in V^\mathbb C$ and $u\in U$.

Note that for the series on the left hand side the actual summation is over the lattice 
$({\widehat {W}\Lambda})^\rho$.
In this context the following formula will be  useful: For every linear automorphism $\widehat{B}$ of $\widehat{U}$ and any lattice $\Lambda$ in $\widehat{U}$ one has
\begin{equation}\label{latticetrans}
\left(\widehat B\Lambda\right)^\rho=\left(\widehat B^\rho\right)^{-1}\Lambda^\rho.
\end{equation}
\begin{theorem}
The identity \eqref{lineartheta}
is equivalent to the identity
\be \label{linearSiegel}
\psi ( Wz) = \widehat{W} \psi(z) \widehat{W}^\rho
\ee
\end{theorem}

\begin{proof} 
Assume that \eqref{linearSiegel} holds. Then, using \eqref{latticetrans} we obtain
\[
\begin{array}{rcl}
\theta_{\widehat {W}\Lambda} (W z, \widehat{W} u) &=& \sum_{l\in(\widehat {W} \Lambda)^\rho }\exp\bigl(i \pi \rho( \psi(W z)l + 2\widehat W  u, l)\bigr)\\
&=& \sum_{l\in(\widehat {W}^\rho)^{-1} \Lambda^\rho }\exp\bigl(i \pi \rho( \widehat W\psi( z)\widehat W^\rho  l + 2\widehat W u, l)\bigr)\\
&=& \sum_{l\in (\widehat {W}^\rho)^{-1} \Lambda^\rho }\exp\bigl(i \pi \rho(\psi( z)\widehat W^\rho l + 2 u,  \widehat W^\rho l)\bigr)\\
&=& \sum_{l^*\in\Lambda^\rho }\exp\bigl(i \pi \rho(\psi( z) l^* + 2 u,  l^*)\bigr)\\
&=& \theta_\Lambda(z,u),
\end{array}
\]
whence \eqref{lineartheta} holds.\\
For the converse direction, consider
\[
\begin{array}{rcl}
\theta_{\widehat {W}\Lambda} (W z, \widehat{W} u)
 &=& \sum_{l\in(\widehat {W} \Lambda)^\rho }\exp\bigl(i \pi \rho( \psi(W z)l + 2\widehat W u, l)\bigr)\\
&=& \sum_{l\in (\widehat {W}^\rho)^{-1} \Lambda^\rho }\exp\bigl(i \pi \rho(\widehat W^{-1}\psi(W  z) (\widehat W^\rho)^{-1}\widehat W^\rho  l + 2 u, 
\widehat W^\rho l)\bigr)\\
&=& \sum_{l^*\in\Lambda^\rho }\exp\bigl(i \pi 
\rho(\widehat W^{-1}\psi(W z) (\widehat W^\rho)^{-1}l^* + 2 u, l^*)\bigr).\\
\end{array}
\]
Comparing the Fourier coefficients of this series to those in the expansion of $\theta_\Lambda(z,u)$, one finds the identity
\[
\exp(\bigl(i \pi \rho(\widehat W^{-1}\psi(W z)(\widehat W^\rho)^{-1}l^* , l^*)\bigr)=
\exp(\bigl(i \pi \rho(\psi(z)l^* , l^*)\bigr)
\]
which implies \eqref{linearSiegel}.
\end{proof}

\begin{remark}\label{trivlin}
Since $(t\cdot I,\,\sqrt t\cdot I)$ obviously satisfies identity \eqref{linearSiegel} for every $t>0$, one has
\[
\theta_{\sqrt t\Lambda}(t\cdot z,\sqrt t\cdot u)=\theta_\Lambda (z,u).
\]
Considering $(I,-I)$ one recovers the fact that $\theta$ is an even function of $u$.
\end{remark}

It turns out that $(W,\widehat W)$ satisfies identity \eqref{linearSiegel} if and only if $(W^\sigma, \widehat{W}^\rho)$ is an automorphism of
$ \mathcal{D} (Y^\sigma, S)$.
Consider the group $\Gamma_0$ of these invertible linear transformations $(W,\widehat W)$.
\begin{theorem} \label{linearauto}
For a pair $ (W, \widehat{W}) $ of automorphisms of $V \times \widehat{U}$ the following statements are equivalent
\begin{itemize}

\item The pair $(W, \widehat{W})$ is contained in $\Gamma_0$.

\item $S( \widehat{W}^\rho u, \widehat{W}^\rho v) = W^\sigma S(u,v)$ for all $u,v \in U$.

\item $ (W^\sigma, \widehat{W}^\rho) $ is an automorphism of $ \mathcal{D} (Y^\sigma, S)$.

\end{itemize}
In particular, for every pair $(W,\widehat{W}) \in \Gamma_0$ the transformation $W$ is an automorphism of the cone $Y$.\\
Moreover, the defining relation for $\Gamma_0$ describes exactly all linear automorphisms of $ \mathcal{D} (Y^\sigma, S)$.

\end{theorem}

\begin{proof}
Assume $( W,\widehat{W} ) \in \Gamma_0$. Then 
\[
\begin{array}{rcl}
\sigma ( S(\widehat{W}^\rho u, \widehat{W}^\rho v),x) &=& \rho(\psi (x) \widehat{W}^\rho u, \widehat{W}^\rho \bar{v} ) = 
\rho (\widehat{W} \psi(x) \widehat{W}^\rho u, \bar{v}) \\
   &=&\rho( \psi(Wx) u, \bar{v}) = \sigma (S(u,v), Wx)
\end{array}
\]
shows equivalence of the first two statements.
To verify the third  claim from the second, let $y \in Y$.  Then $\rho (\psi( W y) l,l) $ is, by the defining property of $Y$,  positive if and only if  
$\rho (\psi(y) \widehat{W}^\rho l,\widehat{W}^\rho l) $  is positive. Hence we obtain $W Y = Y$ and therefore also $ W^\sigma Y^\sigma = Y^\sigma$. As a consequence, for $(z,u) \in  \mathcal{D} (Y^\sigma, S)$ we consider
$(W^\sigma z, \widehat{W}^\rho u)$ and obtain $\Im(W^\sigma z) - S(\widehat W^\rho u,\widehat W^\rho u) = W^\sigma( \Im (z) - S(u,u)) \in W^\sigma Y^\sigma \subset Y^\sigma$. Since $\Gamma_0$ is a group the third claim follows.

The  remaining statements are well known (see e.g. \cite{Do;Habil}, Theorem 1.6).
\end{proof}

\begin{remark}
Considering the relationship between theta series and automorphic forms, transformations that respect lattices are of special relevance.
Thus, with regard to real translations the theta series \eqref{generalthetahat} is invariant only for elements of some  discrete subgroup. \\
For the transformation of theta functions with respect to linear transformations we have a general formula \eqref{lineartheta}, but this formula relates theta functions relative to different lattices, in general.
As a consequence, if one is interested in linear transformations leaving a given theta function invariant, then
in view of the first statement of Theorem \ref{linearauto}, one can only consider those
linear transformations $ (W^\sigma, \widehat{W}^\rho)  \in \Gamma_0$ 
for which the second factor $\widehat{W}$ maps the lattice $\Lambda$ bijectively onto itself.  Therefore the group of second components of linear 
transformations leaving \eqref{generalthetahat} invariant is a discrete subgroup of $GL(\widehat U)$. Since $W$ is uniquely determined by $\widehat W$ whenever $(W,\,\widehat W)\in\Gamma_0$ (see Theorem  \ref{linearauto}, second item and Proposition \ref{properties ofS}), 
the group $\Gamma_0^\Lambda$ of linear theta transformations for which the second factor leaves a given lattice invariant is also discrete. With a little more work, using Proposition \ref{properties ofS}, one can also show that the subgroup of $GL(V)$ formed by the first components of $\Gamma_0^\Lambda$ is discrete.\\
\end{remark}

\section{The full theta transformation formula}
\subsection{Quasi-symmetric domains}
The classical theta function on $\mathbb H\times \mathbb C$ satisfies the important theta trans\-formation formula which may be written as follows:
\[
\theta ( - z^{-1}, z^{-1} u) = \sqrt{-i z}\cdot\exp(i\pi z^{-1}u^2)\cdot \theta(z,u)
\]
This has been generalized for symmetric Siegel domains to a ``full theta trans\-formation formula'' of the type
\be \label{thetafulltransformationjordan}
\theta_{\Lambda^\rho} ( - z^{-1},  \psi (z)^{-1} u) = 
c_\Lambda\cdot H(z,u)\cdot\theta_{\Lambda}  (z,u)
\ee
with a function $H$ that is holomorphic and zero-free on the domain and some constant $c_\Lambda$ depending only on the lattice;  see Dorfmeister \cite{DoInv} and Resnikoff \cite{Res}. In \cite{DoInv} one has the factor (modulo some rewriting involving the Jordan homomorphism $\psi$)
\begin{equation}\label{specialfactor}
H(z,u)= \det(-i\psi(z))^{\frac{1}{2} }\exp({ i \pi \rho ( \psi (z)^{-1} u,u ) } ).
\end{equation}

 In this case, the space $V$ carries the structure of a formally real Jordan algebra (with the base point $e$ as  identity element) and $\psi:V \rightarrow {\rm Sym}(U)$ is a homomorphism from the Jordan algebra $V$ to the special Jordan algebra of $\rho$-symmetric transformations of $\widehat{U}$, i.e. we have 
$\psi(ab) = \frac{1}{2} (\psi(a) \psi(b) + \psi(b) \psi(a) ).$
More generally, these data characterize {\em quasi-symmetric} Siegel domains; see
\cite{Do;AmJMath} and Satake \cite{Satake}.

The main result of \cite{DoInv}, Section 6 may be stated as follows. (Some additional comments have been included in the statement; with regard to the Koecher function and the construction of Jordan algebras we refer to \cite{KoeMinn}, Ch.~I, \S4ff. and Ch.~II, \S5.)

\begin{theorem}
Assume $\mathcal{D} (Y^\sigma, S)$ is a quasi-symmetric Siegel domain. Then we can assume w.l.o.g. that $Y^\sigma = Y$ and $\psi: V \rightarrow {\rm Sym}(U)$ is a homomorphism of the Jordan algebra defined for $Y$ relative to the base point $e$ and the Koecher function $\omega$. Moreover we may assume w.l.o.g. that $\sigma$ is the natural associative positive definite quadratic form associated with $e$ and $\omega$. The theta function associated with these data satisfies (\ref{thetafulltransformationjordan}) with $H$ given by \eqref{specialfactor}.
\end{theorem}
Note that for the classical (symmetric) case as well as for the quasi-symmetric case we have $\psi(z^{-1}) = \psi(z)^{-1}$.

\subsection{General theta functions}\label{genthetasec}

Let us assume now we have a general theta function as defined in \eqref{generalthetahat} on  $\mathcal T(Y),$ satisfying the additional requirements introduced in Section \ref{basicsec}. 

 We want to discuss possible generalizations of the transformation formula \eqref{thetafulltransformationjordan} and to understand what consequences such generalizations imply. A priori we do not have a notion of inverse, but it is possible to generalize properties of $z\mapsto -z^{-1}$ in a natural manner. Indeed, this is an involution of $ \mathcal{H} (Y)$ in the Jordan setting, with the unique fixed point $ie$.\\
Thus, in the general case we first and foremost assume the existence of a non-affine holomorphic involution $j$ from $ \mathcal{H} (Y)$ to itself. By ``involution" we mean that $j$ is involutory and there exists an $e\in Y$ such that the derivative satisfies
\[
dj\left(ie\right)=-I.
\] 
Moreover we assume the existence of a holomorphic map $z\mapsto J(z)$ from $ \mathcal{H} (Y)$ to ${\rm End}\,(U)$, and a holomorphic function $H$ on $\mathcal{H} (Y)$ without zeros, independent of the lattice $\Lambda$, such that the identity
\begin{myequation} \label{theta-transformation}
\theta_{\Lambda^\rho} ( j(z),J(z)u) = 
c_\Lambda\cdot H(z,u)\cdot
\theta_{\Lambda}  (z,u)
\end{myequation}holds with some constant $c_\Lambda$ depending only on $\Lambda$. This seems to be an appropriate generalization of \eqref{thetafulltransformationjordan}, and it is also consistent with the transformation formula in Krieg \cite{KrDiss} , Ch.~IV, Theorem 2.2.  
We normalize $H$ by requiring
\begin{equation}\label{Hnorm}
H(ie,0)=1;
\end{equation}
this makes $c_\Lambda$ unique.
\begin{remark} {The above version of the transformation formula involves two crucial conditions with respect to the function $H$: First, it has no zeros; second, $H$ does not depend on the lattice $\Lambda$. We will require and employ both properties in the following. The question whether these restrictions are (in some way) natural is legitimate. But these properties are satisfied for all classical (including the Jordan) settings. Moreover,  in absence of any further restriction one could choose $j$ and $J$ quite arbitrarily and define a (meromorphic) function $H$ via \eqref{theta-transformation}; but this would be of little interest.}\end{remark}
\begin{remark} 
\begin{enumerate}[(a)]
\item More generally, some authors also discuss {\em partial involutions}, thus $j$ is assumed involutory but may have non-isolated fixed points; see e.g. Krieg \cite{KrDiss}, Dieckmann \cite{Dieck}, Dieckmann and Krieg \cite{DiKr}.
For the case of quaternion half spaces, all partial involutions were determined in \cite{KrDiss}, Ch.~II, Proposition 1.3e), starting from a given involution. In \cite{Dieck}, Kap.~4, partial involutions were determined for the exceptional tube domain. In the present paper we will clarify why only partial involutions can exist in the latter case, but we will not discuss partial involutions in their own right.
\item Kim's \cite{Kim} construction of singular modular forms  on the 27-dimen\-sional exceptional domain is related to partial involutions. As noted by Krieg \cite{KrManu}, these modular forms may be obtained from theta series on a 10-dimen\-sional boundary component; the underlying Jordan algebra  is the (special) algebra of Hermitian $2\times 2$ matrices over the octonions, thus a Peirce-1-space of the exceptional algebra. In turn, Dieckmann \cite{Dieck} Kap.~6 used Kim's and Krieg's results to construct Jacobi forms on the exceptional domain.
\end{enumerate}
\end{remark}

\subsection{The Nullwert}

In this subsection we do not require any properties of $j$, $J$ or of the function $H$ beyond those stated in subsection \ref{genthetasec}. As it turns out, these basic conditions already imply that $j$ is a Jordan inversion.\\
To see this, consider 
\eqref{theta-transformation} for the {\em Nullwert} with $u=0$, which yields
\be \label{theta-j-transformation-null}
\theta_{\Lambda^\rho} ( j(z),0) = 
c_\Lambda\cdot h(z)\cdot
\theta_{\Lambda}  (z,0)
\ee
for all $z\in \mathcal{H} (Y)$, with some holomorphic and zero-free function $h$.\\
Now, for \eqref{theta-j-transformation-null} to make sense for all $z\in \mathcal{H} (Y)$, as is expressly stipulated, one needs convergence of the series on the left hand side. This elementary requirement has strong consequences.
\begin{lemma}
The {\em Nullwert} 
\[
\theta_{\Lambda} (z,0) =  \sum_{l \in \Lambda^\rho} \exp\bigl(\pi i\cdot \rho( \psi(z)l , l)\bigr)
\]
 converges after the substitution $(z,0) \mapsto (j(z), 0)$  for all $z\in \mathcal{H} (Y)$ (if and) only if 
$j$ is a biholomorphic automorphism of the tube domain $ \mathcal{H} (Y)$.
\end{lemma}
\begin{proof}
Substituting the transformation into the theta series, it is necessary and sufficient for convergence that $\rho(\psi(j(z))v,v)$ has a positive imaginary part for all $z$ in $ \mathcal{H} (Y)$, and all nonzero $v\in U$. (Also see  Remark \ref{alternativerem}.) Therefore $j$ maps $ \mathcal{H} (Y)$ into itself and, being an involution, $j$ is an automorphism of the tube domain.
\end{proof}
\begin{theorem}\label{Jordaninverse}
If identity  \eqref{theta-j-transformation-null} holds on $ \mathcal{H} (Y)$ then $ \mathcal{H} (Y)$ is a symmetric tube domain. Hence there exists a formally real Jordan algebra structure $\mathfrak A$ on $V$ with unit element $e$ such that $\sigma$ may be chosen as the trace form of this algebra, $Y=Y^\sigma$ is (for instance) the set of all squares of invertible elements, and $j(z)=-z^{-1}$.
\end{theorem}
\begin{proof} Since $j$ is a biholomorphic automorphism, it is an involutive isometry of $ \mathcal{H} (Y)$ relative to the Bergman metric, with $dj(ie)=-I$. Considering locally the geodesic symmetry $j^*$ relative to the Bergman metric at $ie$ one obtains $j^*\circ j^*={\rm id}$ and $dj^*Ü(ie)=-I$. Since isometries of Riemannian manifolds are uniquely determined by their derivative at some point (see Helgason \cite{Hel}, Ch.~I, Lemma 11.2), one has local equality $j=j^*$. Thus $j^*$ extends to a global biholomorphic map. By Rothaus \cite{Rot}, Theorem 18 the tube domain $ \mathcal{H} (Y)$ is symmetric. The remaining assertions follow, for instance, from Koecher \cite{KoeMinn}.
\end{proof}

\subsection{An incomplete theta transformation} 
In this subsection we explore to which extent the classical proofs of theta transformation formulas can be transferred to a more general setting. We follow the arguments  for the proof of Theorem 2.2 in Resnikoff \cite{Res} (see also Krieg \cite{KrDiss}, Chapter IV).
We first recall a well-known result (see e.g. Lemma 2.1 in  \cite{KrDiss}, Chapter IV). 
 
\begin{lemma}\label{integrallemma} For all $z\in V$ and all $w\in U$ one has
\[
\int_{\widehat U}\exp\left(i\pi\cdot\rho(\psi(z)(v+w),\,v+w)\right)\,{\rm d}v=(\det(-i\psi(z))^{-1/2},
\]
 given the Lebesgue measure on $\widehat U$ for which an orthonormal basis of eigenvectors with respect to $\rho$ spans a box of volume 1. (Any other normalization will yield some nonzero constant factor on the right hand side.)
\end{lemma}

\begin{proof} Since both sides are holomorphic functions
 it suffices to prove the assertion for $z=i\cdot y\in iY$ and $w\in \widehat{U}$. The substitution $x=v+w$ shows that the integral is equal to
\[
\int_{\widehat U}\exp\left(-\pi\cdot\rho(\psi(y)x,\,x)\right)\,{\rm d}x.
\]
Since $\psi(y)$ is self-adjoint and positive definite, a choice of an orthonormal basis of eigenvectors of $\psi(y)$ and the corresponding coordinate transformation will turn this integral into a product of well known one-dimensional integrals.
\end{proof}

We record a simple but useful identity (analogous to ``completion of the square'' in \cite{Res}): 
For all $y\in Y$ and $x,\,v\in U$ one finds
\begin{equation}\label{inverseid}
\begin{array}{rcl}
\rho\left(\psi(y)\left(x+\psi(y)^{-1}v\right),x+\psi(y)^{-1}v\right)&-&\rho\left(\psi(y)^{-1}v,v\right)\\
                                       &=& \rho\left(\psi(y)x+2v,x\right)
\end{array}
\end{equation}
Now we are able to devise a general (albeit ``incomplete") transformation formula:
\begin{theorem}\label{partialtheta} For all $(z,u)\in {\mathcal T}(Y)$ one has
\begin{equation}\label{partialthetaid}
\begin{array}{rcl}
\theta_\Lambda(z,u)&=&C_\Lambda\cdot \det(-i\psi(z))^{-1/2}\cdot \exp\left(-i\pi\rho(\psi(z)^{-1}u,u\right)\\
  &  &\quad  \times\sum_{d\in\Lambda}\exp\left(i\pi\rho(\psi(-z)^{-1}d+2\psi(z)^{-1}u,d)\right)
\end{array}
\end{equation}
with $C_\Lambda$ a positive constant.
\end{theorem}
\begin{proof} Again it suffices to prove the claim for $z=iy\in iY$, and $u\in \widehat U$. For fixed $u$ and $y$, let
\[
f(x)=\exp\left(\pi\rho((-\psi(y)x+2iu,x)\right).
\]
Since $f$ is clearly an $L^1$ function on $\widehat{U}$, the Poisson summation formula (see e.g. Stein and Weiss \cite{StWe}, Appendix) shows 
\[
\theta_\Lambda(iy,u)=
\sum_{l\in\Lambda^\rho}f(l)=
c^*_\Lambda\cdot\sum_{d\in\Lambda}\widehat f(d)
\]
with
\[
\begin{array}{rcl}

\widehat f(d)&=& 
\int_{\widehat U}f(x)\exp\left(-2i\pi\rho(x,d)\right)\,{\rm d}x\\
   &=& 
\int_{\widehat U}\exp\left(\pi\rho\left(-\psi(y)x+2(iu-id),
x\right)\right)\,{\rm d}x
\end{array}
\]
and a positive constant $c_\Lambda^*$.
By identity \eqref{inverseid} with $v=iu-id$ one obtains furthermore
\[
\begin{array}{rcl}
\widehat f(d)&=&\exp\left(\pi\rho\left(\psi(y)^{-1}(iu-id),iu-id\right)\right)\\
   & &\quad\times \int_{\widehat U}f(x)\exp\left(-\pi\rho\left(\psi(y)t,t\right)\right)\,{\rm d}t
\end{array}
\]
with the substitution
\[
t:=x-\psi(y)^{-1}(iu-id).
\]
Lemma \ref{integrallemma} shows that the second factor is (up to some positive constant factor) equal to $\det \psi(y)^{-1/2}$. The first factor is equal to 
\[
\exp\left(-i\pi\rho(\psi(z)^{-1}u,u)\right)\cdot \exp\left(-i\pi\rho\left(\psi(z)^{-1}d -2i\psi(z)^{-1}u, d\right)\right)
\]
after substitution of $y=-iz$. Summing up the series finishes the proof.
\end{proof}
Note that for the case of the upper half plane $\mathbb H$ (and standard choices for $\rho$, $\psi$ and the lattice) this theorem provides a complete proof of the transformation formula. 
\subsection{The main theorem}
We will now employ Theorem \ref{partialtheta} to show that the Jordan algebra $\mathfrak A$ is actually special, and that the transformation $s: \,(z,u)\mapsto (j(z),\,J(z)u)$ is equal to the one given in \eqref{thetafulltransformationjordan}, up to replacing $J$ by its negative. (This ambiguity is inevitable, since the series \eqref{generalthetahat} is even in $u$.)
\begin{theorem}\label{specialjordan} Assume that the theta series \eqref{generalthetahat} satiesfies identities \eqref{theta1} through \eqref{theta-transformation}, and let
$\mathfrak A$ be the Jordan algebra determined in Theorem \ref{Jordaninverse}. Then the identity
\[
\psi(x^{-1})=\psi(x)^{-1}
\]
holds for all invertible $x\in\mathfrak A$; thus $\psi$ is a Jordan homomorphism from $\mathfrak A$ to ${\rm Sym}(\widehat U)$, and $\mathfrak A$ is special. 
Moreover, in the theta transformation formula \eqref{theta-transformation} one has $J(z)=\pm \psi(z)^{-1}$.
\end{theorem}
\begin{proof} \begin{enumerate}[(i)]
\item Combine \eqref{theta-transformation}, Theorem \ref{Jordaninverse} and Theorem \ref{partialtheta} to obtain 
\[
\begin{array}{rcl}
 & & \sum_{d\in\Lambda}\exp\left(i\pi\rho(\psi(-z^{-1})d+ 2 J(z)u,d)\right)\\
 &=&c_\Lambda C_\Lambda\cdot H(z,u)  \det(-i\psi(z))^{-1/2}\cdot \exp\left(-i\pi\rho(\psi(z)^{-1}u,u\right)\\
  &  &\qquad  \times\sum_{d\in\Lambda}\exp\left(i\pi\rho(\psi(-z)^{-1}d+2\psi(z)^{-1}u,d)\right).
\end{array}
\]
Setting $z=ie,\,u=0$ and using \eqref{Hnorm}  one sees 
\[
c_\Lambda C_\Lambda =1
\]
for every lattice.
\item Now let $z=iy$, with $y\in Y =Y^\sigma$ fixed. Then the above identity becomes
\begin{equation}\label{interim1}
\begin{array}{rcl}
& & \sum_{d\in\Lambda}\exp\left(-\pi\rho(\psi(y^{-1})d,d)\right)\exp\left(2i \pi \rho(J(iy)u,d)\right)  \\
  & = &Q(iy, u) \cdot\sum_{d\in\Lambda}\exp\left(-\pi\rho(\psi(y)^{-1}d,d)\right)\exp\left(-2 \pi \rho(\psi(y)^{-1}u,d)\right).
\end{array}
\end{equation}
with 
\[
Q(z,u):=H(z, u) \det(-i\psi(z))^{-1/2}\cdot \exp\left(-i\pi\rho(\psi(z)^{-1}u,u\right)
\]
independent of the lattice $\Lambda$.
Passing from $\Lambda$ to $t \Lambda$ with $t>0$, we obtain
\begin{equation}\label{interim2}
\begin{array}{rcl}
& & \sum_{d\in\Lambda}\exp\left(-t^2\pi\rho(\psi(y^{-1})d,d)\right)\exp\left(2ti \pi \rho(J(iy)u,d)\right)  \\
  & = &Q(iy, u) \cdot\sum_{d\in\Lambda}\exp\left(-t^2\pi\rho(\psi(y)^{-1}d,d)\right)\exp\left(-2t \pi \rho(\psi(y)^{-1}u,d)\right).
\end{array}
\end{equation}
\item We record an auxiliary result. First note that for any positive definite $A\in{\rm Sym}(\widehat U)$, any endomorphism $B$ of $U$ and any fixed $u\in U$ there exists a constant $r>0$ such that 
\[
\begin{array}{rcl}
\exp\left(-t^2\rho(Ad,d)\right)&\leq &\exp(-rt^2\cdot \Vert d\Vert^2);\\
{|{\exp\left(t \rho(Bu,d)\right)}|}&\leq &\exp(rt\cdot \Vert d\Vert)
\end{array}
\]
hold for all $d\in\Lambda$. {Let 
\[
\delta:=\min\left\{\Vert d\Vert;\,d\in\Lambda\setminus\{0\}\right\},\quad R:=r/(2\delta).
\]
Then for all $t\geq 1/(2\delta)$ one has
\[
\exp(-rt^2\cdot \Vert d\Vert^2)\cdot \exp(rt\cdot \Vert d\Vert)\leq \exp(-R\cdot \Vert d\Vert^2)\cdot \exp(R\cdot \Vert d\Vert)
\]
by monotonicity arguments. We claim that for every sequence $(t_k)$ with $t_k\geq 1/(2\delta)$ and $t_k\to\infty$ as $k\to\infty$ one has
\[
\lim_{k\to\infty}\sum_{d\in \Lambda}\exp\left(-t_k^2\,\rho(Ad,d)\right)\exp\left(t_k\, \rho(Bu,d)\right)=1.
\]
This follows from the dominated convergence theorem (see e.g. Bartle \cite{bar}, Theorem 5.6) applied to the counting measure on $\Lambda$, since the estimates above provide a majorant 
\[
g:\,\Lambda\to\mathbb R, \quad d\mapsto \exp\left(-R\Vert d\Vert^2)\right)\exp\left(R\Vert d\Vert\right),
\]
which is integrable; i.e.
$ \sum_{d\in \Lambda}\exp\left(-R\Vert d\Vert^2)\right)\exp\left(R\Vert d\Vert\right) <\infty$.
}Furthermore the estimates show that
\[
\exp\left(-t_k^2\,\rho(Ad,d)\right)\exp\left(t_k\, \rho(Bu,d)\right)\to 0 \quad(k\to\infty)
\]
for all nonzero $d\in\Lambda$.
\item Applying this result to each series in \eqref{interim2}, one finds $Q(iy,u)=1$ for all $y$ and $u$, hence
\begin{equation}\label{interim3}
H(z,u)=\det(-i\psi(z))^{1/2} \cdot \exp({ i \pi \rho ( \psi (z)^{-1} u,u ) } )
\end{equation}
by the identity theorem for holomorphic functions.
\item Reconsidering \eqref{interim1}, one now finds the identity
\[
\begin{array}{rcl}
& &\sum_{d\in\Lambda}\exp\left(i\pi \rho\left(\psi(-z)^{-1}d+2\psi(z)^{-1}u,\,d\right)\right)\\
& &=\sum_{d\in\Lambda}\exp\left(i\pi \rho\left(\psi(-z^{-1})d+2J(z)u,\,d\right)\right).
\end{array}
\]
Letting $v:=\psi(z)^{-1}u$ we have
\begin{equation}\label{interim4}
\begin{array}{rcl}
& &\sum_{d\in\Lambda}\exp\left(i\pi \rho\left(\psi(-z)^{-1}d+2v,\,d\right)\right)\\
&  &=\sum_{d\in\Lambda}\exp\left(i\pi \rho\left(\psi(-z^{-1})d+2J(z)\psi(z)v,\,d\right)\right).
\end{array}
\end{equation}
Since the left hand side is $\Lambda^\rho$-periodic in $v$, so is the right-hand side; hence every linear map $K(z):=J(z)\psi(z)$ is contained in the discrete semigroup which sends $\Lambda^\rho$ to itself. Thus the matrix representation of $K(z)$ with respect to a lattice basis of $\Lambda^\rho$ has integer entries, whence by connectedness and continuity $K(z)=K$ is constant. This property holds for any lattice in $\widehat U$, therefore the representing matrix of $K$ has integer entries with respect to any basis of $\widehat U$, which implies that $K=k\cdot I$ for some integer $k$. We may assume $k\geq 0$ since the series is an even function of $v$.
\item Now rewrite \eqref{interim4} in the form
\[
\begin{array}{rcl}
& &\sum_{d\in\Lambda}\exp\left(i\pi \rho\left(\psi(-z)^{-1}d,\,d\right)\right)\exp\left(2i\pi\cdot \rho\left(v,\,d\right)\right)\\
& &=\sum_{d\in\Lambda}\exp\left(i\pi \rho\left(\psi(-z^{-1})d,\,d\right)\right)\exp\left(2i\pi k \cdot \rho\left(v,\,d\right)\right).
\end{array}
\]
Uniqueness of Fourier coefficients shows $k=1$ and the identities 
\[
 \exp\left(i\pi\rho\left(\psi(-z^{-1})d,\,d\right)\right)=\exp\left(i\pi\rho\left(\psi(-z)^{-1}d,\,d\right)\right),
\]
for all $z$ and $d$.
By continuity arguments one has
\[
\rho\left(\psi(-z^{-1})d,\,d\right)=\rho\left(\psi(-z)^{-1}d,\,d\right)
\]
for all $z$ in a neighborhood of $ie$, and by self-adjointnesss the asserted identity
\[
\psi(-z^{-1})=\psi(-z)^{-1}; \text{  thus  }\psi(z^{-1})=\psi(z)^{-1}
\]
follows. 
\item We finish the proof with a familiar argument: Expanding both sides of
\[
\psi\left((e-w)^{-1}\right)=\left(\psi(e-w)^{-1}\right)=\left(I-\psi(w)\right)^{-1}
\]
for $w$ near $0$ in $V$ by using the geometric series (recall the power-associativity of $\mathfrak A$; see \cite{KoeMinn}). Thus we have
\[
\psi(e+w+w^2+\cdots)=I+\psi(w)+\psi(w)^2+\cdots,
\]
which implies the polynomial identity
\[
\psi(w^2)=\psi(w)^2
\]
for all $w$ in a neighborhood of $0$, and hence on all of $V$. Polarization shows 
\[
2\psi(w_1w_2)=\psi(w_1)\psi(w_2)+\psi(w_2)\psi(w_1)
\]
for all $w_1,\,w_2\in V$, whence $\psi$ is a Jordan homomorphism. Injectivity of $\psi$ implies that $\mathfrak A$ is special.
\end{enumerate}
\end{proof}
One particular consequence of this general description should be mentioned explicitly.
\begin{corollary}
On the exceptional tube domain (which is defined over the exceptional simple formally real Jordan algebra)  there exists no theta series which satisfies identities \eqref{theta1} through \eqref{theta-transformation}.
\end{corollary}

\end{document}